\documentclass[12pt]{amsart}
\usepackage{amssymb}
\usepackage{amsthm}
\usepackage{mathrsfs}
\usepackage{amssymb}
\usepackage{amsfonts}
\usepackage[english]{babel}
\usepackage[T1]{fontenc}
\usepackage[latin1]{inputenc}
\usepackage{fullpage}
\usepackage{amsmath}
\usepackage[all]{xy}
\usepackage{stmaryrd}
\usepackage{amsrefs} 
\usepackage{url}
\usepackage{marginnote}
\usepackage{mathtools}    
\usepackage{color}

\mathtoolsset{showonlyrefs=true}          

\DeclareMathOperator{\Vol}{V}
\newtheorem{theorem}{Theorem}[section]
\newtheorem{lemma}[theorem]{Lemma}
\newtheorem{corollary}[theorem]{Corollary}
\theoremstyle{definition}
\newtheorem{definition}[theorem]{Definition}

\newtheorem{proposition}[theorem]{Proposition}
\newtheorem{conjecture}[theorem]{Conjecture}

\theoremstyle{remark}
\newtheorem{remark}[theorem]{Remark}

\numberwithin{equation}{section}

\def\sideremark#1{\ifvmode\leavevmode\fi\vadjust{\vbox to0pt{\vss
 \hbox to 0pt{\hskip\hsize\hskip1em
 \vbox{\hsize3cm\tiny\raggedright\pretolerance10000
 \noindent #1\hfill}\hss}\vbox to8pt{\vfil}\vss}}}

\begin{document}

\title{Deformations of $Q$-curvature II}

\thanks{Wei Yuan was supported by NSFC (Grant No. 12071489, No. 11521101).}

\author{Yueh-Ju Lin}
\address{(Yueh-Ju Lin) Department of Mathematics, Statistics, and Physics,  Wichita State University, 1845 Fairmount Street, Wichita, KS 67260, USA}
\email{lin@math.wichita.edu}


\author{Wei Yuan}
\address{(Wei Yuan) Department of Mathematics, Sun Yat-sen University, Guangzhou, Guangdong 510275, China}
\email{yuanw9@mail.sysu.edu.cn}




\keywords{$Q$-curvature, volume comparison, Einstein manifolds.}

\begin{abstract}
This is the second article of a sequence of research on deformations of $Q$-curvature. In the previous one, we studied local stability and rigidity phenomena of $Q$-curvature. In this article, we mainly investigate the volume comparison with respect to $Q$-curvature. In particular, we show that volume comparison theorem holds for metrics close to strictly stable positive Einstein metrics. This result shows that $Q$-curvature can still control the volume of manifolds under certain conditions, which provides a fundamental geometric characterization of $Q$-curvature. Applying the same technique, we derive the local rigidity of strictly stable Ricci-flat manifolds with respect to $Q$-curvature, which shows the non-existence of metrics with positive $Q$-curvature near the reference metric.
\end{abstract}
\maketitle




\section{Introduction}
The $Q$-curvature is a $4^{th}$-order scalar type curvature. It has been studied for decades due to its geometric resemblance to Gaussian and scalar curvature as a higher-order curvature quantity.\\

For a closed $4$-dimensional Riemannian manifold $(M^4, g)$, $Q$-curvature is defined to be
\begin{align}\label{Q_4}
Q_g = - \frac{1}{6} \Delta_g R_g - \frac{1}{2} |Ric_g|_{g}^2 + \frac{1}{6} R_g^2.
\end{align}
It satisfies the \emph{Gauss-Bonnet-Chern Formula} 
\begin{align}\label{Gauss_Bonnet_Chern}
\int_{M^4} \left( Q_g + \frac{1}{4} |W_g|^2_g \right) dv_g = 8\pi^2 \chi(M),
\end{align}
where $R_g$, $Ric_g$, and $W_g$ are scalar curvature, Ricci curvature, and Weyl tensor for $(M^4, g)$ respectively. In particular, if $(M^4, g)$ is locally conformally flat,  \emph{i.e.} $W_g = 0$, it reduces to 
\begin{equation}\label{total_Q}
\int_{M^4} Q_g dv_g = 8\pi^2 \chi(M),
\end{equation}
which can be viewed as a generalization of the classic \emph{Gauss-Bonnet Theorem} for closed surfaces.\\

Branson (\cite{Bra85}) extended \eqref{Q_4} and defined the Q-curvature for manifolds with dimension at least three to be
\begin{align}\label{Qgeneral}
Q_{g} = A_n \Delta_{g} R_{g} + B_n |Ric_{g}|_{g}^2 + C_nR_{g}^2,
\end{align}
where $A_n = - \frac{1}{2(n-1)}$ , $B_n = - \frac{2}{(n-2)^2}$ and
$C_n = \frac{n^2(n-4) + 16 (n-1)}{8(n-1)^2(n-2)^2}$. With the aid of the Paneitz operator (\cite{Pan08})
\begin{align}
P_g = \Delta_g^2 - div_g \left[(a_n R_g g + b_n Ric_g) d\right] + \frac{n-4}{2}Q_g,
\end{align}
where $a_n = \frac{(n-2)^2 + 4}{2(n-1)(n-2)}$ and $b_n = - \frac{4}{n-2}$, $Q$-curvature shares a similar conformal transformation law as scalar curvature:
	\begin{align*}
	Q_{\hat g} &= e^{-4u} \left( P_g u + Q_g\right), &\text{for $n=4$ and $\hat g = e^{2u} g$ \   } \\
	Q_{\hat g} &= \frac{2}{n-4} u^{-\frac{n+4}{n-4}} P_g u, &\text{for $n\neq4$ and $\hat g = u^{\frac{4}{n-4}} g$}
	\end{align*}
This suggests that Q-curvature is an $4^{th}$-order analogue of scalar curvature.\\

For a long time, mathematicians are seeking for a better understanding about geometric interpretations of Q-curvature especially in dimensions five and above. In the field of conformal geometry, there has been many excellent works regarding $Q$-curvature (see \cite{HY16} for a great survey). Without the restrictions in conformal classes, there are not many results on $Q$-curvature from the viewpoint of Riemannian geometry so far.\\

The main purpose of the authors' research on $Q$-curvature is to investigate the Riemannian geometric properties of $Q$-curvature. For instances, the authors studied local stability and rigidity phenomena and derived some interesting geometric results about $Q$-curvature (see \cite{LY16, LY17} for more details). These results strongly suggest that $Q$-curvature shares analogous geometric properties as scalar curvature.\\

Volume comparison theorem is a fundamental result in differential geometry. It is important both theoretically and practically in the analysis of geometric problems. The classic volume comparison states that a lower bound for Ricci curvature implies the volume comparison of geodesic balls with those in the model spaces. A natural question is that whether we can replace the assumption on Ricci curvature by a weaker one? As for the scalar curvature, this idea has been proved to be feasible in some special situations (\cite{Yuan20}).\\

Inspired by the second author's work \cite{Yuan20}, we consider the volume comparison for Q-curvature. As the first step, we give the definition of model spaces:

\begin{definition}
A Riemannian manifold $(M^n, \bar{g})$ is \emph{$Q$-critical}, if there is a nontrivial function $f\in C^{\infty}(M)$ and a constant $\kappa\in \mathbb{R}$ such that 
\begin{align}
	\Gamma_{\bar g}^* f = \kappa \bar g,
\end{align}
where $\Gamma_{\bar g}^*: C^\infty (M) \rightarrow S_2 (M)$ is the $L^2$-formal adjoint of $\Gamma_{\bar g}:= DQ_{\bar{g}}$, the linearization of $Q$-curvature at $\bar{g}$. 
\end{definition} 

The concept of $Q$-critical metrics provides a standard model for volume comparison of $Q$-curvature. That is, we only need to consider the volume comparison with respect to $Q$-critical metrics. The reason is that for non-$Q$-critical metrics, one can perturb $Q$-curvature and the volume simultaneously without any constraint:

\begin{theorem}[Case-Lin-Yuan \cite{CLY19}]
Let $(M^n, g)$ be a closed Riemannian manifold. If $(M^n, g)$ is not $Q$-critical, then there are neighborhoods $U$ of Riemannian metrics of $g$ and $V \subset C^{\infty}(M) \oplus \mathbb{R}$ of $(Q_{g}, \Vol_M(g) )$ such that for any $(\psi, \upsilon) \in V$, there is a metric $\hat{g}\in U$ such that $Q_{\hat{g}}= \psi$ and $\Vol_M(\hat{g}) = \upsilon$.
\end{theorem}

\begin{remark}
	 Corvino, Eichmair, and Miao first showed such type of theorem holds for scalar curvature (\cite{CEM13}). For a more general version regarding \emph{conformally variational invariants}, please refer to \cite{CLY19}.\\
\end{remark}

As a special case of $Q$-critical metrics, $J$-Einstein metrics are the most basic examples (\cite{LY17}): 
\begin{definition}\label{def:J_tensor}
Let $(M^n, g)$ be a Riemannian manifold $(n\geq 3)$. We define a symmetric $2$-tensor associated to $Q$-curvature called \emph{$J$-tensor} to be
$$J_{g}:=-\frac{1}{2}\Gamma_{g}^* (1).$$
A metric is called \emph{$J$-Einstein}, if $J_{g}=\Lambda g$ for some constant $\Lambda\in \mathbb{R}$.
\end{definition} 
In particular, if $\bar g$ is Einstein, one can check that
\begin{align}\label{Jtensor}
	J_{\bar g} = \frac{1}{n} Q_{\bar g} \bar g = \frac{(n-2)(n+2)}{8n^2(n-1)^2} R_{\bar g}^2 \bar g.
\end{align}
Therefore, Einstein metrics are $J$-Einstein and hence $Q$-critical. However, under certain conditions, a $J$-Einstein metric has to be Einstein. We present a characterization of Einstein metrics in terms of the spectrum of the Einstein operator 
\begin{align*}
	\Delta_{E}^{\bar{g}}=\Delta_{\bar g} +2Rm_{\bar g}
\end{align*}
defined on the space of symmetric $2$-tensors. 
\begin{theorem}\label{thm:rigidity_J_Einstein}
	Suppose $(M^n, \bar g)$ is an $n$-dimensional closed $J$-Einstein manifold and the Einstein operator $\Delta_E^{\bar g}$ on $S_2(M)$ satisfies
	\begin{align*}
	\Lambda_E^{\bar g} := \inf_{h \in S_2(M)\setminus \{0\}} \frac{\int_M \langle h, - \Delta_E^{\bar g} h \rangle_{\bar g} dv_{\bar g}}{\int_M |h|_{\bar g}^2 dv_{\bar g}} > - \frac{(n-2)^3(n+2)}{8n(n-1)^2} \min_M R_{\bar g}.
	\end{align*}
	Furthermore, we assume the scalar curvature $R_{\bar g}$ is a constant when $3 \leq n\leq 8$.
	Then $\bar g$ is an Einstein metric.
\end{theorem} 

\vskip.1in

This result is an important motivation for one to take Einstein metrics to be reference metrics when considering the volume comparison with respect to $Q$-curvature. 
Unfortunately, being an Einstein metric is not sufficient for volume comparison to hold. In fact, one needs to impose a stronger assumption on Einstein metrics.

\begin{definition}[Stability of Einstein manifolds \cite{Bes87, Kro14}]\label{stableEindef}

For $n\geq 3$, suppose $(M^n, \bar g)$ is a closed Einstein manifold.
The Einstein metric $\bar g$ is said to be \emph{strictly stable}, if the \emph{Einstein operator}
	\begin{align*}
	\Delta_{E}^{\bar{g}}=\Delta_{\bar g} +2Rm_{\bar g}
	\end{align*}	
is a negative operator on $S_{2, \bar g}^{_{TT}}(M) \backslash \{0\}$, where
$$S_{2, \bar g}^{_{TT}}(M):=\{h\in S_2(M)\ |\ \delta_{\bar{g}}h=0,\ tr_{\bar{g}} h=0\}$$ is the space of transverse-traceless symmetric 2-tensors on $(M^{n}, \bar{g})$.\\
\end{definition}

Now we state our main result in this article, which concerns a volume comparison with respect to $Q$-curvature for closed strictly stable Einstein manifolds.
\begin{theorem}\label{thmvolQ} 
For $n\geq 3$, suppose $(M^{n},\bar{g})$ is an $n$-dimensional closed strictly stable Einstein manifold with Ricci curvature
$$Ric_{\bar{g}} = (n-1)\lambda\bar{g},$$
where $\lambda >0$ is a constant. 
Then there exists a constant $\varepsilon_0 > 0$ such that for any metric $g$ on $M$ satisfying $$Q_{g}\geq Q_{\bar{g}}$$ and $$||g - \bar{g}||_{C^4(M,\bar{g})}<\varepsilon_0,$$ the following volume comparison holds 
$$\Vol_M(g) \leq \Vol_M(\bar {g}),$$ with the equality holds if and only if $g$ is isometric to $\bar{g}$.
\end{theorem}

\begin{remark}\label{rmk:no_vol_comp_ricci_flat}
	The above volume comparison does not hold for Ricci flat metrics. This is easy to see by taking $g =c^2\bar{g}$ for some constant $c\neq 0$. Clearly, the $Q$-curvature is $Q_{g} = Q_{\bar g} = 0$, but the volume $\Vol_M(g)$ can be either larger or smaller than $\Vol_M(\bar{g})$ depending on $c > 1$ or $c < 1$.
\end{remark}

\begin{remark}
The strictly stability condition in Theorem \ref{thmvolQ} is the same as the scalar curvature case (see \cite{Yuan20}).  It is also necessary for $Q$-curvature. See Remark \ref{stab_ness} for more details regarding a counterexample. \\
\end{remark}

In particular, Theorem \ref{thmvolQ} implies the volume comparison for metrics near the spherical metric, since the reference metric is strictly stable. This is a $4^{th}$-order analogue of \emph{Bray's conjecture} for scalar curvature (see \cite{Yuan20} for more details).
\begin{corollary}
	For $n\geq 3$, let $(\mathbb{S}^{n},\bar{g})$ be the canonical sphere with Ricci curvature
	$$Ric_{\bar{g}} = (n-1)\bar{g}.$$ 
	Then there exists a constant $\varepsilon_0 > 0$ such that for any metric $g$ on $\mathbb{S}^n$ satisfying $$Q_{g}\geq \frac{1}{8}n(n-2)(n+2)$$ and $$||g - \bar{g}||_{C^4(M,\bar{g})}<\varepsilon_0,$$ the following volume comparison holds 
	$$\Vol_M(g) \leq V_{\mathbb{S}^n},$$ with the equality holds if and only if $g$ is isometric to $\bar{g}$.
\end{corollary}
\ \\

	According to Remark \ref{rmk:no_vol_comp_ricci_flat}, the volume comparison for Ricci-flat manifolds can not be expected. However, applying the same idea as proof of Theorem \ref{thmvolQ}, we can show that strictly stable Ricci-flat manifolds admits local rigidity with respect to $Q$-curvature. This extends our previous local rigidity result for tori and answers the question proposed by the referee of our earlier article \cite{LY16}.
	\begin{theorem}\label{thm:ricci_flat_rigidity}
		Suppose $(M^n ,\bar g)$ is a strictly stable Ricci-flat manifold, then there exists a constant $\varepsilon_0 > 0$ such that any metric $g$ satisfying
		\begin{align*}
		Q_g \geq 0
		\end{align*}
		and
		\begin{align*}
		||g - \bar g||_{C^4(M, \bar g)} < \varepsilon_0
		\end{align*}
implies $g$ has to be Ricci-flat.	In particular, there is no metric with positive $Q$-curvature near $\bar g$.
	\end{theorem}

	\begin{remark}
		It is not difficult to improve this local rigidity result by a weaker assumption that the Ricci-flat metric $\bar g$ is only stable instead. It would be interesting to ask whether we can find an example of unstable Ricci-flat manifold which admits a metric of positive Q-curvature.
	\end{remark}
\ \\

This article is organized as follows: In Section \ref{notation}, we introduce the notation and useful formulas needed throughout the article. In Section \ref{varJ}, we show the rigidity of Einstein metrics in the category of $J$-Einstein metrics and calculate the first variation of $J$-tensor at Einstein metrics which will be used in Section \ref{volcom}. In Section \ref{varfunc}, we calculate the variational formulas for the main functional. In Section \ref{volcom}, we prove our main results (Theorem \ref{thmvolQ}, \ref{thm:ricci_flat_rigidity}) by showing nonpositivity of second variation of the functional and using Morse lemma argument.
In Section \ref{example}, we provide a counterexample showing that the strictly stability of Einstein metric is necessary for our main result. We also make some observations about a global volume comparison of $Q$-curvature for a locally conformally flat $4$-manifold. \\

\paragraph{\textbf{Acknowledgement}}
The authors would like to express their appreciations to Professors Jeffrey S. Case and Yen-Chang Huang for many inspiring discussions. We would like to thank Professor Yoshihiko Matsumoto for introducing his remarkable work \cite{Mat13} and Professor Mijia Lai for a valuable comment on Corollary \ref{cor:global_vol_comp_4-sphere_hyperbolic}. Yueh-Ju Lin would also like to thank Princeton University for the support, as part of the work was done when she was in Princeton.\\


\section{Preliminary and notation}\label{notation}

\subsection{Notations}
Throughout this article, we will always assume $(M^n, g)$ to be an $n$-dimensional closed Riemannian manifold ($n \geq 3$) unless otherwise stated. Also, we list notations involved in this article:

$\mathcal{M}$ - the set of all smooth metrics on $M$;

$\mathscr{D}(M)$ - the set of all smooth diffeomorphisms $ \varphi : M \rightarrow M$;

$\mathscr{X}(M)$ - the set of all smooth vector fields on $M$;

$S_2(M)$ - the set of all smooth symmetric 2-tensors on $M$;

$\Vol_M(g)$ - the volume of manifold $M$ with respect to the metric $g$.\\

We adopt the following convention for Ricci curvature tensor
\begin{align*}
R_{jk} = R^i_{ijk} = g^{il} R_{ijkl}.
\end{align*}
and denote its traceless part as 
\begin{align*}
	E_g : = Ric_g - \frac{1}{n} R_g g.
\end{align*}
The Schouten tensor is defined to be
\begin{align*}
	S_{g}=\frac{1}{n-2}\left(Ric_{g}-\frac{1}{2(n-1)}R_{g}g \right)
\end{align*}
and $\mathring S_g$ is denoted to be its traceless part.\\

For Laplacian operator, we use the convention as follows
\begin{align*}
\Delta_g := g^{ij} \nabla_i \nabla_j.
\end{align*}

\vskip.1in

For simplicity, we introduce following operations:
\begin{align*}
(h \times k )_{ij} := g^{kl}h_{ik}k_{jl} = h_i^lk_{lj}, \ \ \ \ \ \ \ 
h \cdot k  := tr_g (h \times k) = g^{ij}g^{kl}h_{ik}k_{jl} = h^{jk}k_{jk}
\end{align*}
and 
\begin{align*}
(Rm \cdot h )_{jk}:= R_{ijkl} h^{il}
\end{align*}
for any $h, k \in S_2(M)$.\\

Let $X \in \mathscr{X}(M)$ and $h \in S_2(M)$, we use following notations for the operator
\begin{align*}
(\delta_g h)_i := - (div_g h)_i = -\nabla^j h_{ij},
\end{align*}
which is the $L^2$-formal adjoint of Lie derivative (up to a scalar multiple) $$\frac{1}{2}(L_g X)_{ij} = \frac{1}{2} ( \nabla_i X_j + \nabla_j X_i).$$
The Einstein operator acting on $h\in S_2(M)$ is defined to be
\begin{align*}
\Delta_E^{g} h = \Delta_g h + 2 Rm_g\cdot
h .
\end{align*}

The $J$-tensor (\cite{LY17}) is defined to be 
\begin{align}
J_g = \frac{1}{n} Q_g g - \frac{1}{n-2} B_g - \frac{n-4}{4(n-1)(n-2)} T_g,
\end{align}
where
\begin{align*}
B_g 
=& \Delta_E^g \mathring S_g  - \nabla^2_g (tr_{g}S_{g}) + \frac{1}{n} g \Delta_g (tr_{g}S_{g}) - (n -4) \mathring S_g^2  - |\mathring S_g|_g^2 g  - \frac{2(n-2)}{n} (tr_{g}S_{g}) \mathring S_g
\end{align*}
is the \emph{Bach tensor} and
\begin{align*}
\hspace{1.8em}T_g :=& (n-2) \left( \nabla^2_g (tr_g S_g) - \frac{1}{n} g\Delta_g (tr_g S_g) \right) + 4(n-1) \left( \mathring S_g^2 - \frac{1}{n} |\mathring S_g|_g^2 g\right) \\
&- \frac{(n-2)(n^2 + 2n -4)}{n}  (tr_g S_g) \mathring{S}_g.
\end{align*}

\subsection{Basic variational formulae}\label{sec:basic_var_form}
\vskip.1in
We list several formulas for linearization of geometric quantities that will be useful for later sections (see \cite{FM75, LY16, Yuan20} for detailed calculations). \\

The linearization of Ricci tensor is
\begin{align*}
(DRic_g) \cdot h =  - \frac{1}{2}\left[ \Delta_E^{g} h - (Ric_g \times h + h \times Ric_g)+ 
\nabla^2_{g} (tr_g h) + (\nabla_j (\delta_g h)_k + \nabla_k (\delta_g
h)_j)dx^{j}\otimes dx^{k}  \right],
\end{align*}
and the linearization of scalar curvature is
\begin{align*}
(DR_g) \cdot h =  - \Delta_g (tr_g h) +  \delta^2_g h - Ric_g \cdot h.
\end{align*}
The linearization of $Q$-curvature is
\begin{align*}
\Gamma_g h :=& (DQ_g) \cdot h  \\
=& A_n \left[ - \Delta^2_g (tr_g h) +  \Delta_g \delta_g^2 h  -
\Delta_g ( Ric_g \cdot h ) + \frac{1}{2} dR_g \cdot (d( tr_g h ) + 2\delta_g h) - \nabla_g^2 R_g\cdot h\right]\\
 & - B_n \left[  Ric_g \cdot \Delta_E^g h + Ric_g \cdot
\nabla^2_g(tr_g h) + 2 Ric_g \cdot\nabla (\delta_g h) 
\right] \\
&+ 2 C_n R_g \left[ - \Delta_g (tr_g h) +  \delta_g^2 h - Ric_g \cdot h
\right].
\end{align*}
and the $L^2$-formal adjoint of $\Gamma_g$ is 
\begin{align*}
	\Gamma_g^* f :=& A_n \left[ \nabla^2_g
	\Delta_g f - g \Delta^2_g f  - Ric_g \Delta_g f + \frac{1}{2} g \delta_g (f dR_g) + \nabla_g ( f
	dR_g) - f \nabla^2_g R_g \right]\\ \notag
	& - B_n \left[ \Delta_g (f Ric_g) + 2 f
	(Rm_g\cdot Ric_g) + g \delta^2_g (f Ric_g) + 2 \nabla_g \delta_g (f
	Ric_g) \right]\\ \notag
	&- 2 C_n \left[ g\Delta_g (f R_g) - \nabla^2_g (f R_g) + f R_g
	Ric_g \right],
	\end{align*}
where $A_n, B_n , C_n$ are defined in \eqref{Qgeneral}.
The first and second variations of the volume functional are 
\begin{align*}
(D\Vol_{M,g}) \cdot  h = \frac{1}{2} \int_{M} (tr_{g}h) dv_{g}
\end{align*}
and 
\begin{align*}
(D^2\Vol_{M,g}) \cdot  (h, h)= \frac{1}{4} \int_{M} [(tr_{g}h)^2- 2|h|^2_{g}] dv_{g}.
\end{align*}
\ \\


\section{$J$-tensor and Einstein metrics}\label{varJ}

In this section, we will discuss some involved topics about $J$-tensor and Einstein metrics.

\subsection{Rigidity of Einstein metrics in the category of $J$-Einstein metrics}
\ \\

As we have stated in the introduction, Einstein metrics can be identified with a characterization in the spectrum of Einstein operator in the category of $J$-Einstein metrics. Now we present a simple proof here.
 
\begin{proof}[Proof of Theorem \ref{thm:rigidity_J_Einstein}]
	By definition, the $J$-Einstein metric $\bar g$ satisfies the equation
	\begin{align*}
	\mathring J_{\bar g} =& - \frac{1}{n-2}\left(B_{\bar g} + \frac{n-4}{4(n-1)} T_{\bar g} \right)\\
	=&-\frac{1}{n-2}\left[ \Delta_E^{\bar g} \mathring S_{\bar g} + \frac{n^2-10n+12}{4(n-1)}\left( \nabla^2_{\bar g} (tr_{\bar g} S_{\bar g}) - \frac{1}{n} \bar g \Delta_{\bar g}(tr_{\bar g}S_{\bar g}) \right)  \right] + \frac{2}{n}|\mathring S_{\bar g}|^2_{\bar g} \bar g\\
	&+ \frac{(n-2)^2(n+2)}{4n(n-1)}(tr_{\bar g}S_{\bar g})\mathring S_{\bar g} \\
	=& 0.
	\end{align*}
	It implies
	
	\begin{align*}
	0 =& \int_M \langle \mathring J_{\bar g}, \mathring S_{\bar g} \rangle_{\bar g} dv_{\bar g} \\
	=& \frac{1}{n-2}\int_M \left[ - \left\langle \mathring S_{\bar g}, \Delta^{\bar g}_E \mathring S_{\bar g}  \right\rangle_{\bar g} + \frac{n^2-10n+12}{4n}  |d (tr_{\bar g} S_{\bar g})|^2_{\bar g} + \frac{(n-2)^3(n+2)}{8n(n-1)^2} R_{\bar g} |\mathring{S}_{\bar g}|^2_{\bar g} \right] dv_{\bar g}\\
	\geq& \frac{1}{n-2} \left( \Lambda_E^{\bar g} + \frac{(n-2)^3(n+2)}{8n(n-1)^2}\min_M R_{\bar g}\right) \int_M |\mathring{S}_{\bar g}|^2_{\bar g} dv_{\bar g}\\
	\geq& 0
	\end{align*}
	due to assumptions on the spectrum of $(-\Delta_E^{\bar g})$ and 
	\begin{align*}
		tr_{\bar g} S_{\bar g} = \frac{ R_{\bar g}}{2(n-1)}
	\end{align*}
	is a constant for $3\leq n \leq 8$. Therefore, we conclude that $$E_{\bar g} = (n-2) \mathring S_{\bar g} = 0,$$
	which shows $\bar g$ is an Einstein metric. 
\end{proof}

\vskip .2in

In particular, we have

\begin{corollary}
		Suppose $(M^n, \bar g)$ is a closed $J$-Einstein manifold with non-negative constant scalar curvature and the Einstein operator $\Delta_E^{\bar g}$ on $S_2(M)$ is negative, then $\bar g$ is a strictly stable Einstein metric.
\end{corollary}

\begin{proof}
	It is straightforward that $\bar g$ is an Einstein metric according to Theorem \ref{thm:rigidity_J_Einstein}. Moreover, we obtain
	\begin{align*}
	\inf_{h \in S_{2,\bar g}^{_{TT}}(M)\setminus \{0\}} \frac{\int_M \langle h, - \Delta_E^{\bar g} h \rangle_{\bar g} dv_{\bar g}}{\int_M |h|_{\bar g}^2 dv_{\bar g}} \geq \inf_{h \in S_2(M)\setminus \{0\}} \frac{\int_M \langle h, - \Delta_E^{\bar g} h \rangle_{\bar g} dv_{\bar g}}{\int_M |h|_{\bar g}^2 dv_{\bar g}} > 0,
	\end{align*}
	since $S_{2,\bar g}^{_{TT}} \subsetneqq S_2(M)$. Thus the metric $\bar g$ is strictly stable Einstein by definition.
\end{proof}

\vskip.3in

\subsection{Variations of $J$-tensor at Einstein metrics}
\ \\

As a geometric symmetric $2$-tensor, $J$-tensor has a natural connection with $Q$-curvature as we have discussed in the introduction. It is crucial to investigate variational properties of $J$-tensor, when considering variational problems associated to $Q$-curvature. 
In this section, we will obtain the first variation of the traceless part of $J$-tensor at an Einstein metric along the $TT$-direction, which is critical in our further discussion.\\

For simplicity, we may use $'$ to denote the first variation in the space of metrics. The direction of the variation will be clear from the context.

\begin{lemma}\label{lem:var_E}
	Suppose $\bar g$ is an Einstein metric, then 
	\begin{align*}
	(DE_{\bar g}) \cdot \mathring h = -\frac{1}{2} \Delta_E^{\bar g} \mathring h
	\end{align*}
	and
	\begin{align*}
	(DR_{\bar g}) \cdot \mathring h = 0
	\end{align*}
	for any $\mathring h \in S_{2,\bar g}^{_{TT}} (M)$.
\end{lemma}

\begin{proof}
	This is straightforward from first variations of Ricci and scalar curvature in Section \ref{sec:basic_var_form} together with facts that $E_{\bar g} = 0$ and $\mathring h  \in S_{2,\bar g}^{_{TT}}(M)$.
\end{proof}

\vskip.1in

For Einstein metrics, the connection Laplacian is commutative with first variation:
\begin{lemma}\label{lem:var_Delta_E}
	Suppose $\bar g$ is an Einstein metric, then
	\begin{align*}
	(\Delta_{\bar g} E)' =& \Delta_{\bar g} E'.
	\end{align*}
\end{lemma}

\begin{proof}
	It is straightforward that	
	\begin{align*}
	&\Delta_{\bar g} E_{jk} \\
	=& \bar g^{il} \nabla_i \nabla_l E_{jk} \\
	=& \bar g^{il} \left[\partial_i (\nabla_l E_{jk}) - \Gamma^p_{il}\nabla_p E_{jk} - \Gamma^p_{ij}\nabla_l E_{pk} - \Gamma^p_{ik}\nabla_l E_{jp} \right] \\
	=& \bar g^{il} \left[\partial_i (\partial_l E_{jk} - \Gamma_{lj}^p E_{pk} - \Gamma_{lk}^p E_{jp}) - \Gamma^p_{il}\nabla_p E_{jk} - \Gamma^p_{ij}\nabla_l E_{pk} - \Gamma^p_{ik}\nabla_l E_{jp} \right] \\
	=& \bar g^{il} \left[\partial_i \partial_l E_{jk} - (\partial_i \Gamma_{lj}^p) E_{pk} - (\partial_i \Gamma_{lk}^p) E_{jp}  -  \Gamma_{lj}^p\partial_i E_{pk} -  \Gamma_{lk}^p\partial_i E_{jp}- \Gamma^p_{il}\nabla_p E_{jk} - \Gamma^p_{ij}\nabla_l E_{pk} - \Gamma^p_{ik}\nabla_l E_{jp} \right].
	\end{align*}
	This shows
	\begin{align*}
	&(\Delta_{\bar g} E_{jk})' \\
	=& \bar g^{il} \left[\partial_i \partial_l E_{jk}' - (\partial_i \Gamma_{lj}^p) E_{pk}' - (\partial_i \Gamma_{lk}^p) E_{jp}'  -  \Gamma_{lj}^p\partial_i E_{pk}' -  \Gamma_{lk}^p\partial_i E_{jp}' - \Gamma^p_{il}\nabla_p E_{jk}' - \Gamma^p_{ij}\nabla_l E_{pk}' - \Gamma^p_{ik}\nabla_l E_{jp}' \right] \\
	=& \Delta_{\bar g} E_{jk}',
	\end{align*}
	when $\bar g$ is Einstein.
\end{proof}

\vskip.1in

Now the first variation of Bach tensor at an Einstein metric $\bar g$ is given as follows:
\begin{proposition}
	Suppose $\bar g$ is an Einstein metric, then for any $\mathring h \in S_{2, \bar g}^{_{TT}}(M)$, 
	\begin{align}\label{eqn:first_var_Bach_TT}
	(DB_{\bar g}) \cdot \mathring h =& -\frac{1}{2(n-2)}\left(-\Delta_E^{\bar g} + \frac{n-2}{n(n-1)} R_{\bar g} \right)(- \Delta_E^{\bar g}  \mathring h).
	\end{align}
\end{proposition}

\begin{proof}
	Rewriting the Bach tensor in terms of scalar curvature and traceless Ricci tensor, we have
	\begin{align*}
	B_g 
	=&\frac{1}{n-2}\Delta_E^g E_g  - \frac{1}{2(n-1)} \left( \nabla^2_g R_g - \frac{1}{n} g \Delta_g R_g\right) - \frac{n -4}{ (n-2)^2} E_g^2  - \frac{ 1}{(n-2)^2}|E_g|_g^2 g - \frac{R_g}{n(n-1)} E_g.
	\end{align*}
	From Lemma \ref{lem:var_Delta_E} and $E_{\bar g}=0$, 
		\begin{align*}
		B_{\bar g}' 
		=&\frac{1}{n-2}\Delta_E^{\bar g} E_{\bar g}'  - \frac{1}{2(n-1)} \left( \nabla^2_{\bar g} R_{\bar g}' - \frac{1}{n} {\bar g} \Delta_{\bar g} R_{\bar g}'\right)  - \frac{R_{\bar g}}{n(n-1)} E_{\bar g}', 
		\end{align*}	
	where all variations are taken along the $TT$-direction $\mathring h \in S_{2, \bar g}^{_{TT}}(M)$. According to Lemma \ref{lem:var_E}, we have $R_{\bar g}' = 0$ and hence
	\begin{align*}
	(DB_{\bar g}) \cdot \mathring h =& -\frac{1}{2(n-2)}\left(-\Delta_E^{\bar g} + \frac{n-2}{n(n-1)} R_{\bar g} \right)(-\Delta_E^{\bar g}  \mathring h).
	\end{align*}
\end{proof}

\begin{remark}
	Note that when $n=4$ and
	\begin{align*}
	Ric_{\bar g} = 3 \lambda \bar g,
	\end{align*}
	then equation \eqref{eqn:first_var_Bach_TT} becomes
	\begin{align*}
	(DB_{\bar g}) \cdot \mathring h= -\frac{1}{4}\left(-\Delta_E^{\bar g} +2 \lambda \right) (-\Delta_E^{\bar g} \mathring h).
	\end{align*}	
	Recall that Bach tensor appears to be the obstruction tensor in the study of ambient space (\cite{GH05}). In fact, Matsumoto has shown that for $2m^{th}$-order obstruction tensor $\mathcal{O}_{\bar g}^{(2m)}$, its first variation at a $2m$-dimensional Einstein metric with Ricci curvature $$Ric_{\bar g}=(2m-1)\lambda \bar{g}$$ is given by
	\begin{align}
\left(D\mathcal{O}_{\bar g}^{(2m)}\right) \cdot \mathring h=& \frac{(-1)^{m-1}}{4(m-1)} \left[\prod_{k= 0}^{m-1}\left(-\Delta_E^{\bar g} + 2k(2m-2k-1) \lambda \right) \right] \mathring h,
	\end{align}
	for any $\mathring h \in S_{2, \bar g}^{_{TT}}(M)$. Please see \cite{Mat13} for further information.
\end{remark}

\vskip.1in

For the rest part of $\mathring J_{\bar g}$, we have
\begin{lemma}
	Suppose $\bar g$ is an Einstein metric, then
	\begin{align*}
	(DT_{\bar g})\cdot \mathring h=  \frac{n^2 + 2n -4}{4n(n-1)}  R_{\bar g} \Delta_E^{\bar g} \mathring h
	\end{align*}
	for any $\mathring h\in S_{2, \bar g}^{_{TT}}(M)$.
\end{lemma}

\begin{proof}
	Rewriting the tensor $T_g$ in terms of scalar curvature $R_g$ and traceless Ricci curvature $E_g$, we get
	\begin{align*}
	T_g =& \frac{n-2}{2(n-1)} \left( \nabla^2_g R_g - \frac{1}{n} g \Delta_g R_g \right) + \frac{4(n-1)}{(n-2)^2} \left( E_g^2  - \frac{1}{n} | E_g |_g^2 g\right) - \frac{n^2 + 2n -4}{2n(n-1)}  R_g E_g.
	\end{align*}
	Then the result follows from the fact that $E_{\bar g} = 0$ and Lemma \ref{lem:var_E}.
\end{proof}

\vskip.1in

Now combining first variations of $B_g$ and $T_g$, we derive
\begin{proposition}\label{prop:second_var_mathring_J}
	Suppose $\bar g$ is an Einstein metric, then
	\begin{align*}
	(D\mathring J_{\bar{g}}) \cdot \mathring h =&  \frac{1}{2(n-2)^2}\left(-\Delta_E^{\bar g} + \frac{(n-2)^3(n + 2)}{8n(n-1)^2}  R_{\bar{g}} \right) (-\Delta_E^{\bar g}  \mathring h),
	\end{align*}	
	for any $\mathring h \in S_{2,\bar g}^{_{TT}}(M)$. \\
\end{proposition}

According to Definition \ref{stableEindef}, immediately we obtain the following characterization of the operator $D \mathring J_{\bar g}$ for strictly stable Einstein metrics $\bar g$:
\begin{corollary}\label{cor:stable_Einstein_DJ_positive}
	Suppose $\bar g$ is an Einstein metric, then $D \mathring{J}_{\bar g}$ is a formally self-adjoint operator on $S_{2,\bar g}^{_{TT}}(M)$. Moreover, $D \mathring{J}_{\bar g}$ is a positive operator, if $\bar g$ is a strictly stable Einstein metric with non-negative scalar curvature.
\end{corollary}

\ \\

\section{The key functional and its variations} \label{varfunc}

For completeness, we start with generic $J$-Einstein metrics instead of more restricted Einstein metrics.\\

Suppose $(M^n, \bar g)$ is a closed $J$-Einstein manifold and consider the functional
\begin{align}
	\mathcal{F}_{M,\bar g} (g) = \Vol_M(g)^{\frac{4}{n}} \int_M Q(g) dv_{\bar g}.
\end{align}
Note that the volume form $dv_{\bar{g}}$ is independent of $g$ and thus the functional $\mathcal{F}_{M,\bar g}$ is \emph{scaling invariant}:
\begin{align}
	\mathcal{F}_{M,\bar g} ( c^2 g) = \mathcal{F}_{M,\bar g} (g)
\end{align}
for any real number $c \neq 0$.\\

The functional $\mathcal F_{M, \bar g}$ is particularly designed for our purpose. This can be glimpsed from its variational properties.
\begin{proposition}\label{prop:critical_pt_J_Einstein}
	The $J$-Einstein metric $\bar g$ is a critical point of the functional $\mathcal{F}_{M,\bar g}$.
\end{proposition}

\begin{proof}
	For any $h \in S_2(M)$, 
	\begin{align*}
		(D \mathcal{F}_{M,\bar g}) \cdot h =& \Vol_M(\bar g)^{\frac{4}{n}} \int_M \left((D Q_{\bar g}) \cdot h\right) dv_{\bar g} + \frac{4}{n}\Vol_M(\bar g)^{\frac{4}{n} - 1} \Vol'_{M, \bar g} \int_M Q_{\bar g} dv_{\bar g}\\
		=& \Vol_M(\bar g)^{\frac{4}{n}} \left[ \int_M \langle \Gamma_{\bar g} h, 1 \rangle_{\bar g} dv_{\bar g} + \frac{2}{n} \left(\int_M (tr_{\bar g} h) dv_{\bar g}\right) \Vol_M(\bar g)^{ - 1} \int_M Q_{\bar g} dv_{\bar g} \right]\\
		=& \Vol_M(\bar g)^{\frac{4}{n}} \left[ \int_M \langle  h, \Gamma_{\bar g}^*1 \rangle_{\bar g} dv_{\bar g} + \frac{2}{n} \overline{Q_{\bar g}} \int_M (tr_{\bar g} h) dv_{\bar g}  \right]\\
		=& -2 \Vol_M(\bar g)^{\frac{4}{n}}  \int_M \left\langle  h, J_{\bar g} - \frac{1}{n} \overline{Q_{\bar g}} \bar g \right\rangle_{\bar g} dv_{\bar g} 
	\end{align*}
	where 
	\begin{align*}
		\overline{Q_{\bar g}} :=  \Vol_M(\bar g)^{ - 1} \int_M Q_{\bar g} dv_{\bar g}
	\end{align*}
	is the average of $Q_{\bar g}$ on $M$.
	
	According to Definition \ref{def:J_tensor}, a $J$-Einstein metric has constant $Q$-curvature and hence 
	\begin{align*}
		J_{\bar g} - \frac{1}{n} \overline{Q_{\bar g}} \bar g = 0,
	\end{align*}
	which implies $\bar g$ is a critical point of $\mathcal F_{M, \bar g}$.
\end{proof}

In principle, one can obtain the second variation of $\mathcal F_{M, \bar g}$ from a formal calculation. However, due to its complicated expression, the calculation would be rather messy. Instead, with an elegant trick, we can minimize the work through the first variation of $J$-tensor. Analogous to the previous section, we adopt the convention that $'$ stands for first and $''$ for second variations with certain $h\in S_2(M)$.\\

We start with a useful observation:
\begin{lemma}\label{lem:tr_circ_J'}
	For any metric $g$ and $h \in S_2(M)$,
	\begin{align}
		tr_g \mathring{J}_g'= \langle \mathring J_g, \mathring h\rangle_g,
	\end{align}
where $\mathring{J}_g'$ is the first variation of traceless $J$-tensor and $\mathring h$ is the traceless part of $h$.
\end{lemma}

\begin{proof}
	Since $tr_g \mathring{J}_g = 0$, differentiating both sides of the equation gives
	\begin{align*}
		0 = tr_g \mathring{J}_g' + (g^{-1})'\cdot \mathring{J}_g = tr_g \mathring{J}_g' - h\cdot \mathring{J}_g.
	\end{align*}
	That is,
	\begin{align*}
		tr_g \mathring{J}_g'= \langle \mathring J_g, h\rangle_g = \left\langle \mathring J_g, \mathring h + \frac{1}{n}(tr_g h) g \right\rangle_g = \langle \mathring J_g, \mathring h \rangle_g.
	\end{align*}	
\end{proof}

With the above observation, we can express the integral of second variation of $Q$-curvature in a compact form.
\begin{proposition}\label{prop:second_variations_integral_Q}
	Suppose $\bar g$ is a $J$-Einstein metric, then we have
	\begin{align*}
	\int_M Q_{\bar g}'' dv_{\bar g} =
	&-2  \int_M \left[ \langle \mathring{J}_{\bar g}', \mathring{h} \rangle_{\bar g}  - \frac{1}{n} Q_{\bar g} |\mathring{h}|_{\bar g}^2 + \left(  \frac{n+4}{4n} Q_{\bar g}'  + \frac{n-2}{2n^2} Q_{\bar g} (tr_{\bar g} h) \right) (tr_{\bar g} h) \right] dv_{\bar g}
	\end{align*}
	holds for any $h\in S_2(M)$.	
\end{proposition}

\begin{proof}
	We start with an arbitrary Riemannian metric $g$ on $M$. It is straightforward that for any $h \in S_2(M)$, 
	\begin{align*}
	&\left( \int_M Q_g dv_g \right)'' \\
	=& \left( \int_M Q_g' dv_g + \int_M Q_g \left(dv_g \right)' \right)'\\
	=& \left( \int_M \langle \Gamma_g^*1, h\rangle_g dv_g \right)' + \int_M Q_g' \left(dv_g \right)' + \int_M Q_g \left(dv_g \right)''\\	
	=& -2  \int_M\left[ \langle J_g', h\rangle_g + 2\left\langle \left(g^{-1}\right)', J_g\times h \right \rangle_g + \frac{1}{2} \langle J_g, h\rangle_g (tr_g h)\right] dv_g + \int_M Q_g' \left(dv_g \right)' + \int_M Q_g \left(dv_g \right)''\\
	=& -2  \int_M\left[ \langle J_g', h\rangle_g - 2\left\langle J_g, h^2 \right \rangle_g + \frac{1}{2} \langle J_g, h\rangle_g (tr_g h)\right] dv_g + \int_M Q_g' \left(dv_g \right)' + \int_M Q_g \left(dv_g \right)''.
	\end{align*}
	Furthermore, the decomposition $$J_g = \mathring{J}_g + \frac{1}{n} Q_g g$$ yields $$J_g' = \mathring{J}_g' + \frac{1}{n} Q_g' g + \frac{1}{n} Q_g h.$$ 
	Together with variational formulas in Section \ref{sec:basic_var_form}, we have
	\begin{align*}
	&\int_M Q_g'' dv_g \\
	=&\left( \int_M Q_g dv_g \right)'' - 2 \int_M Q_g' \left(dv_g \right)' - \int_M Q_g \left(dv_g \right)'' \\
	=& -2 \int_M \left[\left\langle \mathring{J}_g' + \frac{1}{n} Q_g' g + \frac{1}{n} Q_g h, h \right\rangle_g -2 \left\langle \mathring{J}_g + \frac{1}{n} Q_g g, h^2 \right \rangle_g \right] dv_g \\
	&-  \int_M \left[ \left\langle \mathring{J}_g + \frac{1}{n} Q_g g, h \right \rangle_g +\frac{1}{2}Q'_{g}\right](tr_g h) dv_g \\
	=& -2  \int_M \left[ \langle \mathring{J}_g', h \rangle_g -2 \left\langle \mathring{J}_g + \frac{1}{2n} Q_g g, h^2 \right \rangle_g+ \frac{1}{2}\left( \frac{n+4}{2n} Q_g'  + \frac{1}{n} Q_g (tr_g h) + \langle \mathring{J}_g , h \rangle_g \right)(tr_g h) \right] dv_g.
	\end{align*}
	From the decomposition $h = \mathring{h} + \frac{1}{n}(tr_g h) g$ and Lemma \ref{lem:tr_circ_J'}, we can further simplify the expression to the following form:
	\begin{align*}
	&\int_M Q_g'' dv_g \\
	=& -2  \int_M \left[ \left\langle \mathring{J}_g', \mathring{h} + \frac{1}{n}(tr_g h) g \right\rangle_g -2 \left\langle \mathring{J}_g + \frac{1}{2n} Q_g g, \mathring{h}^2 + \frac{2}{n} (tr_g h) \mathring{h} + \frac{1}{n^2}(tr_g h)^2 g \right \rangle_g \right] dv_g \\
	&- \int_M \left[\frac{n+4}{2n} Q_g'  + \frac{1}{n} Q_g (tr_g h) + \left\langle \mathring{J}_g , \mathring{h} + \frac{1}{n}(tr_g h) g \right \rangle_g \right](tr_g h) dv_g\\
	=& -2  \int_M \left[ \langle \mathring{J}_g', \mathring{h} \rangle_g -2 \left\langle \mathring{J}_g + \frac{1}{2n} Q_g g, \mathring{h}^2  \right \rangle_g \right] dv_g \\
	&-\int_M \left[  \frac{n+4}{2n} Q_g'  + \frac{(n-2)}{n^2} Q_g (tr_g h) + \frac{(n-6)}{n} \langle \mathring{J}_g , \mathring{h}\rangle_g \right](tr_g h) dv_g.
	\end{align*}
	 
	In particular for $J$-Einstein $\bar g$, we have $\mathring{J}_{\bar g}=0$ and hence
	\begin{align*}
	\int_M Q_{\bar g}'' dv_{\bar g} 
	=& -2  \int_M \left[ \langle \mathring{J}_{\bar g}', \mathring{h} \rangle_{\bar g}  - \frac{1}{n} Q_{\bar g} |\mathring{h}|_{\bar g}^2 + \left(  \frac{n+4}{4n} Q_{\bar g}'  + \frac{n-2}{2n^2} Q_{\bar g} (tr_{\bar g} h) \right) (tr_{\bar g} h) \right] dv_{\bar g}.
	\end{align*}
\end{proof}
\vskip.1in

Now we calculate the second variation of the functional $\mathcal{F}_{M,\bar g}$ at a $J$-Einstein metric $\bar g$.
\begin{proposition}\label{svarF} 
For $n\geq 3$, suppose $(M^{n},\bar{g})$ is $J$-Einstein, then for any $h\in S_2(M)$ we have
\begin{align*}
&\Vol_M(\bar g)^{-\frac{4}{n}} \left((D^2 \mathcal{F}_{M,\bar g}) \cdot (h, h)\right) \\
=&-2  \int_M  \left[ \langle \mathring h , (D \mathring{J}_{\bar g}) \cdot \mathring{h} \rangle_{\bar g}+ \frac{1}{n} \left( tr_{\bar g} ( (D \mathring{J}_{\bar g})^{*} \cdot \mathring{h}) + \frac{n+4}{4} (\Gamma_{\bar g} \mathring h)\right) (tr_{\bar g} h) \right]  dv_{\bar g} \\
&-\frac{n+4}{2n^2} \int_M  \left[ ( tr_{\bar g} h - \overline{tr_{\bar g} h})   \mathscr{L}_{\bar g} (tr_{\bar g} h - \overline{tr_{\bar g} h}) \right] dv_{\bar g},
\end{align*}
where the operator $\mathscr{L}_{\bar g}$ is defined to be
\begin{align*}
	\mathscr{L}_{\bar g} u := tr_{\bar g} \Gamma_{\bar g}^* u = \frac{1}{2} \left( P_{\bar g} - \frac{n+4}{2} Q_{\bar g} \right) u.
\end{align*}
\end{proposition}

\begin{proof}
	Applying variational formulae of volume in Section \ref{sec:basic_var_form}, we have
	\begin{align*}
	&(D^2 \mathcal{F}_{M,\bar g}) \cdot (h, h) \\
	=& \left(\Vol_M(g)^{\frac{4}{n}} \int_M  Q_g dv_{\bar g} \right)''\\
	=& \left( (\Vol_M(g)^{\frac{4}{n}})' \int_M  Q_g dv_{\bar g} + \Vol_M(g)^{\frac{4}{n}} \int_M  Q_g' dv_{\bar g}  \right)'\\
	=& \Vol_M(\bar g)^{\frac{4}{n}} \int_M Q_{\bar g}'' dv_{\bar g} + \frac{8}{n}\Vol_M(\bar g)^{\frac{4}{n} - 1} \Vol_{M,\bar g}' \int_M  Q_{\bar g}' dv_{\bar g} +  \frac{4}{n}\Vol_M(\bar g)^{\frac{4}{n} - 1} \Vol_{M,\bar g}'' \int_M Q_{\bar g} dv_{\bar g} \\
	&- \frac{4(n - 4)}{n^2}\Vol_M(\bar g)^{\frac{4}{n} - 2}  ( \Vol_{M, \bar g}')^2 \int_M Q_{\bar g} dv_{\bar g}\\
	=& \Vol_M(\bar g)^{\frac{4}{n}} \int_M Q_{\bar g}'' dv_{\bar g} - \frac{8}{n}\Vol_M(\bar g)^{\frac{4-n}{n}} \int_M (tr_{\bar g} h) dv_{\bar g} \int_M \langle J_{\bar g}, h \rangle_{\bar g} dv_{\bar g}\\
	&+  \frac{1}{n}\Vol_M(\bar g)^{\frac{4}{n}} Q_{\bar g} \left[  \int_M \left( \frac{n-2}{n}(tr_{\bar g} h)^2 - 2|\mathring h|_{\bar g}^2 \right)dv_{\bar g} - \frac{n - 4}{n} \Vol_M(\bar g)^{-1} \left(\int_M (tr_{\bar g} h) dv_{\bar g}\right)^2  \right],
	\end{align*}
	where we used the fact $Q_{\bar g}$ is a constant, when $\bar g$ is $J$-Einstein. 
	
	Now applying Lemma \ref{prop:second_variations_integral_Q}, it can be simplified as
	\begin{align*}
	&\Vol_M(\bar g)^{-\frac{4}{n}} \left((D^2 \mathcal{F}_{M,\bar g}) \cdot (h, h)\right) \\
	=& -2  \int_M \langle \mathring{J}_{\bar g}', \mathring{h} \rangle_{\bar g} dv_{\bar g}  - \frac{n+4}{2n} \int_M  Q_{\bar g}' (tr_{\bar g} h) dv_{\bar g} - \frac{n+ 4}{n^2}Q_{\bar g} \Vol_M(\bar g)^{ - 1} \left(\int_M (tr_{\bar g} h) dv_{\bar g} \right)^2.
	\end{align*}	
	Rewriting its first term as
	\begin{align*}
		-2  \int_M \langle \mathring{J}_{\bar g}', \mathring{h} \rangle_{\bar g} dv_{\bar g} =& -2  \int_M \left \langle (D \mathring{J}_{\bar g}) \cdot \left(\mathring h + \frac{1}{n} (tr_{\bar g} h) \bar g \right), \mathring{h} \right\rangle_{\bar g} dv_{\bar g} \\
		=& -2  \int_M \langle \mathring h , (D \mathring{J}_{\bar g}) \cdot \mathring{h} \rangle_{\bar g} dv_{\bar g} - \frac{2}{n}  \int_M  [tr_{\bar g} ( (D \mathring{J}_{\bar g})^{*} \cdot \mathring{h})](tr_{\bar g} h) dv_{\bar g}, 
	\end{align*}
	we obtain
		\begin{align*}
		&\Vol_M(\bar g)^{-\frac{4}{n}} \left((D^2 \mathcal{F}_{M,\bar g}) \cdot (h, h)\right) \\
		=&-2  \int_M \langle \mathring h , (D \mathring{J}_{\bar g}) \cdot \mathring{h} \rangle_{\bar g} dv_{\bar g}   - \frac{2}{n}  \int_M  [tr_{\bar g} ( (D \mathring{J}_{\bar g})^{*} \cdot \mathring{h})](tr_{\bar g} h) dv_{\bar g} - \frac{n+4}{2n} \int_M \left[  Q_{\bar g}' (tr_{\bar g} h) \right]dv_{\bar g} \\
&- \frac{n+ 4}{n^2}Q_{\bar g} \Vol_M(\bar g)^{ - 1} \left(\int_M (tr_{\bar g} h) dv_{\bar g} \right)^2\\
		=&-2  \int_M \langle \mathring h , (D \mathring{J}_{\bar g}) \cdot \mathring{h} \rangle_{\bar g} dv_{\bar g}   - \frac{2}{n}  \int_M  [tr_{\bar g} ( (D \mathring{J}_{\bar g})^{*} \cdot \mathring{h})](tr_{\bar g} h) dv_{\bar g} - \frac{n+4}{2n} \int_M \left[ (tr_{\bar g} h) (\Gamma_{\bar g} \mathring h) \right] dv_{\bar g} \\
&- \frac{n+4}{2n^2} \int_M   ( tr_{\bar g} h) [tr_{\bar g}\Gamma_{\bar g}^* (tr_{\bar g} h)] dv_{\bar g} - \frac{n+ 4}{n^2}Q_{\bar g} \Vol_M(\bar g)^{ - 1} \left(\int_M (tr_{\bar g} h) dv_{\bar g} \right)^2.
		\end{align*}
	Denote 
$$\mathring P_{\bar g} := P_{\bar g} - \frac{n-4}{2} Q_{\bar g}$$ to be the divergence part of Paneitz operator. Clearly the space of constants are contained in $\ker \mathring P_{\bar g}$. This implies 

	\begin{align*}
		&- \frac{n+4}{2n^2} \int_M   ( tr_{\bar g} h) [tr_{\bar g}\Gamma_{\bar g}^* (tr_{\bar g} h)] dv_{\bar g}  - \frac{n+ 4}{n^2}Q_{\bar g} \Vol_M(\bar g)^{ - 1} \left(\int_M (tr_{\bar g} h) dv_{\bar g} \right)^2\\
		=& - \frac{n+4}{4n^2} \int_M   ( tr_{\bar g} h) \left[ \left( \mathring P_{\bar g} - 4 Q_{\bar g} \right) (tr_{\bar g} h) \right] dv_{\bar g}  - \frac{n+ 4}{n^2}Q_{\bar g} \Vol_M(\bar g)^{ - 1} \left(\int_M (tr_{\bar g} h) dv_{\bar g} \right)^2\\
		=& - \frac{n+4}{4n^2} \int_M   ( tr_{\bar g} h) \left(  \mathring P_{\bar g} (tr_{\bar g} h) \right) dv_{\bar g} + \frac{n+ 4}{n^2}Q_{\bar g} \int_M (tr_{\bar g} h - \overline{tr_{\bar g} h})^2 dv_{\bar g} \\
		=& - \frac{n+4}{4n^2} \int_M  \left[ ( tr_{\bar g} h - \overline{tr_{\bar g} h}) \left(  \mathring P_{\bar g} (tr_{\bar g} h - \overline{tr_{\bar g} h}) \right) - 4 Q_{\bar g} (tr_{\bar g} h - \overline{tr_{\bar g} h})^2\right] dv_{\bar g} \\
		=& - \frac{n+4}{2n^2} \int_M  \left[ ( tr_{\bar g} h - \overline{tr_{\bar g} h})   \mathscr{L}_{\bar g} (tr_{\bar g} h - \overline{tr_{\bar g} h}) \right] dv_{\bar g} .
	\end{align*}
	Therefore,	
\begin{align*}
&\Vol_M(\bar g)^{-\frac{4}{n}} \left((D^2 \mathcal{F}_{M,\bar g}) \cdot (h, h)\right) \\
=&-2  \int_M  \left[ \langle \mathring h , (D \mathring{J}_{\bar g}) \cdot \mathring{h} \rangle_{\bar g} + \frac{1}{n} [tr_{\bar g} ( (D \mathring{J}_{\bar g})^{*} \cdot \mathring{h})](tr_{\bar g} h) + \frac{n+4}{4n} (tr_{\bar g} h)(\Gamma_{\bar g} \mathring h) \right]  dv_{\bar g} \\
&- \frac{n+4}{2n^2} \int_M  \left[ ( tr_{\bar g} h - \overline{tr_{\bar g} h})   \mathscr{L}_{\bar g} (tr_{\bar g} h - \overline{tr_{\bar g} h}) \right] dv_{\bar g}.\\
=&-2  \int_M  \left[ \langle \mathring h , (D \mathring{J}_{\bar g}) \cdot \mathring{h} \rangle_{\bar g} + \frac{1}{n} \left( tr_{\bar g} ( (D \mathring{J}_{\bar g})^{*} \cdot \mathring{h}) + \frac{n+4}{4} (\Gamma_{\bar g} \mathring h)\right) (tr_{\bar g} h) \right]  dv_{\bar g} \\
&- \frac{n+4}{2n^2} \int_M  \left[ ( tr_{\bar g} h - \overline{tr_{\bar g} h})   \mathscr{L}_{\bar g} (tr_{\bar g} h - \overline{tr_{\bar g} h}) \right] dv_{\bar g}.
\end{align*}
\end{proof}
\vskip.1in

\begin{remark}
In general, for a $J$-Einstein metric $\bar g$, the term 
$$\int_M [tr_{\bar g} ( (D \mathring{J}_{\bar g})^{*} \cdot \mathring{h})](tr_{\bar g} h) dv_{\bar g}$$ in the above proposition may not necessarily vanish unless the dimension $n=4$. But if $\bar g$ is an Einstein metric, this term would be zero due to the operator $D\mathring{J}_{\bar g}$ is formally self-adjoint and Lemma \ref{lem:tr_circ_J'}. This case is addressed in Corollary \ref{cor:secvarEin}.
\end{remark}
\ 

In order to remove the crossing term in the previous proposition and obtain an elegant variational formula for Einstein metrics, we first rewrite the operator $\Gamma_g$ in a nice way.  
\begin{proposition} 
	Suppose $g$ is an arbitrary Riemannian metric and 
	\begin{align*}
		h = \mathring h + \frac{1}{n}(tr_g h) g \in  S^{_{TT}}_{2,g}(M) \oplus (C^\infty(M) \cdot g),
	\end{align*}
	then
	\begin{align*}
		\Gamma_g h =& div_g [U_g(\mathring h)] - 2 \mathring J_g \cdot \mathring h + \frac{1}{n} \mathscr{L}_g (tr_g h) ,
	\end{align*}
	where 
		\begin{align*}
		U_g (\mathring h) := \frac{1}{2(n-1)(n-2)} \left[ n^2   \nabla_{g} (\mathring S_{g} \cdot \mathring h)  - 8 (n - 1 )\mathring h_{ij} \nabla_g \mathring S_g^{ij} +(n^2 + 4n - 8 ) \mathring h ( \nabla_g (tr_{g} S_{g})) \right].
		\end{align*}
\end{proposition}

\begin{proof}
	By the relation between Ricci and Schouten tensor $$Ric_{g}=(n-2)S_g + (tr_g S_g)g,$$ we can express $\Gamma_g \mathring h$ in terms of Schouten tensor
	\begin{align*}
		\Gamma_g \mathring h 
		=& \frac{n-2}{2(n-1)}\Delta_g(\mathring S_g  \cdot \mathring h ) + \nabla^2 (tr_g S_g)\cdot \mathring h +\frac{2}{n-2} \mathring S_g \cdot \Delta_g \mathring h +\frac{4}{n-2} \mathring S_g \cdot ( Rm_g \cdot \mathring h)\\
		&  - \frac{(n-2)^2(n+2)}{2n(n-1)} (tr_g S_g) (\mathring S_g  \cdot \mathring h),
	\end{align*} 
	which has simpler coefficients instead.
	Recall the traceless part of $J$-tensor is given by
	\begin{align*}
		\mathring J_g =&- \frac{1}{n-2} \left( B_g + \frac{n-4}{4(n-1)} T_g \right)\\
	=&-\frac{1}{n-2}\left( \Delta_{g} \mathring S_g + \frac{n^2-10n+12}{4(n-1)}\left( \nabla^2_g (tr_g S_g) - \frac{1}{n}g\Delta_{g}(tr_{g}S_{g}) \right) + 2 Rm_g \cdot \mathring S_g \right) \\
	&+ \frac{(n-2)^2(n+2)}{4n(n-1)}(tr_{g}S_{g})\mathring S_{g} + \frac{2}{n}|\mathring S_g|^2_g g .
	\end{align*}		
	Now we have
	\begin{align*}
		&\Gamma_g \mathring h + 2 \mathring J_g \cdot \mathring h\\
		=& \frac{n-2}{2(n-1)}\Delta_g(\mathring S_g  \cdot \mathring h ) + \frac{n^2 + 4n - 8  }{2(n-1)(n-2)}\nabla^2 (tr_g S_g)\cdot \mathring h +\frac{2}{n-2} \left[ \mathring S_g \cdot \Delta_g \mathring h -  ( \Delta_g \mathring S_g ) \cdot \mathring h \right]\\
		=& \frac{n-2}{2(n-1)}div_g \left[\nabla_g(\mathring S_g  \cdot \mathring h )  + \frac{n^2 + 4n - 8 }{(n-2)^2} \mathring h ( \nabla_g (tr_g S_g))  +\frac{4(n-1)}{(n-2)^2} \left( \mathring S_g^{ij} \nabla_g \mathring h_{ij} - \mathring h_{ij} \nabla_g \mathring S_g^{ij}   \right) \right]\\	
		=& \frac{1}{2(n-1)(n-2)} div_g \left[  n^2 \nabla_g ( \mathring S_g \cdot \mathring h )  - 8(n -1 )\mathring h_{ij} \nabla_g \mathring S_g^{ij} +(n^2 + 4n - 8 ) \mathring h ( \nabla_g (tr_g S_g)) \right]\\
		=& div_g [U_g (\mathring h)],
	\end{align*}
	where
	\begin{align*}
		U_g (\mathring h) = \frac{1}{2(n-1)(n-2)} \left[ n^2 \nabla_g ( \mathring S_g \cdot \mathring h )  - 8(n -1 )\mathring h_{ij} \nabla_g \mathring S_g^{ij} +(n^2 + 4n - 8 ) \mathring h ( \nabla_g (tr_g S_g)) \right]
	\end{align*}
	That is, 
	\begin{align*}
		\Gamma_g \mathring h = div_g[U_g (\mathring h)] - 2 \mathring J_g \cdot \mathring h.
	\end{align*}

	As for the trace part, we have
	\begin{align*}
		\int_M \left[ u \ \Gamma_g \left( \frac{1}{n} (tr_g h)g \right) \right] dv_g =& \int_M \left \langle \Gamma_g^* u, \  \frac{1}{n} (tr_g h)g \right \rangle_g  dv_g = \frac{1}{n} \int_M \left[ (tr_g h) (tr_g \Gamma_g^* u) \right]   dv_g
	\end{align*}
	for any $u \in C^\infty(M)$. Recall that
	\begin{align*}
		\mathscr{L}_{g} = tr_g \Gamma_g^*
	\end{align*}
	as defined in Proposition \ref{svarF} and $\mathscr{L}_{g}$ is formally self-adjoint, then
	\begin{align*}
				\int_M \left[ u \ \Gamma_g \left( \frac{1}{n} (tr_g h)g \right) \right] dv_g = \frac{1}{n} \int_M \left[ (tr_g h) (\mathscr{L}_g u) \right] dv_g = \frac{1}{n} \int_M \left[ u\ \mathscr{L}_g (tr_g h) \right] dv_g.
	\end{align*}
	Since $u$ is arbitrary, we conclude that
		
	\begin{align*}
		\Gamma_g \left( \frac{1}{n} (tr_g h)g \right) = \frac{1}{n} \mathscr{L}_g (tr_g h).
	\end{align*}
	
	Combining these two parts, we obtain
	\begin{align*}
		\Gamma_g h 
		=& div_g [U_g(\mathring h)] - 2 \mathring J_g \cdot \mathring h + \frac{1}{n} \mathscr{L}_g (tr_g h).
	\end{align*}	
\end{proof}

\vskip .2in

In particular, when $\bar g$ is restricted to be an Einstein metric, the second variation of $\mathcal{F}_{M,\bar g}$ can be expressed in an elegant way:
\begin{corollary}\label{cor:secvarEin}
	Suppose $(M^{n}, \bar g)$ is an Einstein manifold, then 
	\begin{align*}
		&(D^2 \mathcal{F}_{M,\bar g}) \cdot (h, h) \\
		=&-2 \Vol_M(\bar g)^{\frac{4}{n}}  \left[ \int_M \langle \mathring h , (D \mathring{J}_{\bar g}) \cdot \mathring{h} \rangle_{\bar g} dv_{\bar g} + \frac{n+4}{4n^2} \int_M  \left[ ( tr_{\bar g} h - \overline{tr_{\bar g} h})   \mathscr{L}_{\bar{g}} (tr_{\bar g} h - \overline{tr_{\bar g} h}) \right] dv_{\bar g} \right],
	\end{align*}	
for any $h = \mathring h + \frac{1}{n}(tr_{\bar g} h) \bar{g} \in S^{_{TT}}_{2,\bar g}(M) \oplus (C^\infty(M)\cdot \bar{g})$.
\end{corollary}

\begin{proof} 
	This result follows from Corollary \ref{cor:stable_Einstein_DJ_positive}, Lemma \ref{lem:tr_circ_J'}, and the fact that $U_{\bar g} (\mathring h) = 0$ and $\mathring J_{\bar g} = 0$ when $\bar g$ is an Einstein metric.
\end{proof}

\ \\

\section{Volume comparison with respect to $Q$-curvature} \label{vol_comp_Q}\label{volcom}

In this section, we give the proof of our main results. As the first step, we recall some fundamental results involved (c.f. \cite{Bes87, Via13}):

\begin{lemma}[Lichnerowicz-Obata's eigenvalue estimate]\label{lem:L-O}
Suppose $(M^{n},\bar{g})$ is an $n$-dimensional closed Riemannian manifold with 
$$Ric_{\bar g} \geq (n-1)\lambda \bar{g},$$
where $\lambda>0$ is a constant. Then for any function $u\in C^{\infty}(M)\backslash \{0\}$ with
$$\int_{M} u dv_{\bar{g}} =0,$$
we have 
$$ \int_{M} |du|^2 dv_{\bar{g}} \geq n\lambda \int_{M} u^2 dv_{\bar{g}},$$ where equality holds if and only if $(M^{n}, \bar{g})$ is isometric to the round sphere $\mathbb{S}^{n}(r)$ with radius
$r=\frac{1}{\sqrt{\lambda}}$.
\end{lemma}

\vskip .2in

\begin{lemma}[Berger-Ebin's splitting lemma for Einstein manifolds] 
Suppose $(M^{n}, \bar g)$ is an $n$-dimensional closed Einstein manifold with Ricci curvature $$Ric_{\bar g}=(n-1)\lambda\bar{g},$$ then we have the direct sum decomposition 
$$S_2(M) = \text{Im}\ \mathcal{L}_{\bar g} \oplus (C^\infty (M) \cdot \bar{g})  \oplus S_{2,\bar g}^{_{TT}}(M)$$
unless $(M^{n}, \bar g)$ is isometric to the round sphere $\mathbb{S}^{n}(r)$ up to a scaling, where $\mathcal{L}_{\bar g}$ is the Lie derivative.
For the round sphere $\mathbb{S}^{n}(r)$ with radius $r=\frac{1}{\sqrt{\lambda}}$, we have 
$$S_2(M) = \text{Im}\ \mathcal{L}_{\bar g} \oplus (E^{\perp}_{n\lambda} \cdot \bar{g})  \oplus S_{2,\bar g}^{_{TT}}(M),$$
where $$E_{n\lambda}:= \{u\in C^{\infty}(\mathbb{S}^{n}(r)) \ |\ \Delta_{\mathbb{S}^{n}(r)}u + n\lambda u =0\}$$
is the space of first eigenfunctions for the spherical metric and $E^{\perp}_{n\lambda}$ is its $L^2$-orthogonal complement.
\end{lemma}

With the aid of \emph{Implicit Function Theorem}, \emph{Berger-Ebin's splitting lemma} suggests that one can define a concept named \emph{local slice $\mathcal{S}_{\bar g}$}, which is very helpful in understanding the local structure of Einstein metrics in $\mathcal{M}$. Simply speaking, a local slice is a set of equivalent classes of metrics near the reference metric $\bar g$ modulo diffeomorphisms. The process of pulling back metrics on the local slice is also known as \emph{gauge fixing}. \\

The above splitting lemma does not provide an orthogonal decomposition despite being a direct sum decomposition. To overcome this issue, we need a refined decomposition which involves the splitting of vector fields as well. As a result, we obtain the following improved version of the traditional \emph{Ebin-Palais slice theorem}. The proof is very similar to the traditional one (see \cite{BM11, Via13}). It seems like such a treatment did not appear in the literature to the best of our knowledge and we hope it would benefit researches of similar topics. 
\begin{theorem}[Ebin-Palais slice theorem] \label{thm:slice_thm}
	Suppose $(M^n,\bar{g})$ is a closed $n$-dimensional Einstein manifold with Ricci curvature tensor
	\begin{align*}
	Ric_{\bar g} = (n-1) \lambda \bar g,
	\end{align*}
	where $\lambda \in \mathbb{R}$ is a constant. There exists a local slice $\mathcal{S}_{\bar{g}}$ though $\bar g$ in $\mathcal{M}$. That is, for a fixed real number $p > n$, one can find a constant $\varepsilon_1 > 0$ such that for any metric $g \in \mathcal{M}$ with $||g-\bar{g}||_{W^{2,p}(M,\bar{g})} < \varepsilon_1$, there is a diffeomorphism $\varphi\in \mathscr{D}(M)$ with $\varphi^*g \in \mathcal{S}_{\bar{g}}$. Moreover, for a smooth local slice $\mathcal{S}_{\bar{g}}$, we have the decomposition
	$$S_2(M) =T_{\bar g} \mathcal{S}_{\bar g} \oplus (T_{\bar g} \mathcal{S}_{\bar g})^{\perp},$$ where the tangent space of $\mathcal{S}_{\bar g}$ at $\bar g$ and its $L^2$-orthogonal complement are given by
	$$T_{\bar g}\mathcal{S}_{\bar g} = S_{2,\bar g}^{_{TT}}(M) \oplus (C^\infty (M) \cdot \bar g)$$
	and 
	$$(T_{\bar g}\mathcal{S}_{\bar g})^{\perp}= \left\{\mathcal{L}_{\bar g}(X)\ | \ \langle X, \nabla_{\bar g} u \rangle_{L^2(M, \bar g)} = 0, \ \forall u\in C^{\infty}(M) \right\},$$
	when $(M^n, \bar g)$ is not isometric to the round sphere $\mathbb{S}^{n}(r)$ up to a scaling;
	$$T_{\bar g}\mathcal{S}_{\bar g} =  S_{2,\bar g}^{_{TT}}(M) \oplus (E_{n\lambda}^\perp  \cdot \bar g)$$
	and 
	$$(T_{\bar g} \mathcal{S}_{\bar g})^{\perp}= \left\{\mathcal{L}_{\bar g}(X)\ | \ \langle X, \nabla_{\bar g} u \rangle_{L^2(M, \bar g)} = 0, \ \forall u\in E_{n\lambda}^\perp \right\},$$
	when $(M^n,\bar{g})$ is isometric to the round sphere $\mathbb{S}^{n}(r)$ with $r=\frac{1}{\sqrt{\lambda}}$. Here
	$$E_{n\lambda} = \{ u \in C^\infty(\mathbb{S}^n(r)): \Delta_{\mathbb{S}^n(r)} u + n\lambda u = 0 \}$$ is the space of first eigenfunctions for the spherical metric.
\end{theorem}

\vskip .2in

In order to estimate the second variation of $\mathcal{F}_{M,\bar g}$, we also need to investigate the analytic properties of the operator $\mathscr{L}_{\bar g}$ (as defined in Proposition \ref{svarF}):
\begin{proposition}\label{prop:positive_L}
	Suppose $\bar g$ is an Einstein metric with Ricci curvature
	\begin{align*}
		Ric_{\bar g} = (n-1) \lambda \bar g,
	\end{align*}
	where $\lambda \geq 0$ is a constant, then the operator $\mathscr{L}_{\bar g}$ is a non-negative operator. Moreover, $\mathscr{L}_{\bar g}$ admits non-trivial kernel when 
	\begin{itemize}
		\item $\lambda > 0$ and $\bar g$ is spherical: $$\ker \mathscr{L}_{\bar g} = E_{n\lambda},$$
		\item $\lambda = 0$ and $\bar g$ is Ricci-flat: $$\ker \mathscr{L}_{\bar g} = \mathbb{R}.\quad $$
	\end{itemize}
\end{proposition}

\begin{proof}
	By definition,
	\begin{align*}
		\mathscr{L}_{\bar g}u =& \frac{1}{2} \left( P_{\bar g} - \frac{n+4}{2} Q_{\bar g} \right) u
		= \frac{1}{2} \left( -\Delta_{\bar g} + \frac{(n-2)(n+2)}{2} \lambda \right) \left( -\Delta_{\bar g} - n \lambda\right)u .
	\end{align*}
	
For $\lambda > 0$, the first eigenvalue of $(-\Delta_{\bar g})$ is at least $n\lambda$ by Lemma \ref{lem:L-O}, which implies the operator $\mathscr{L}_{\bar g}$ is non-negative and 
	\begin{align*}
		\ker \mathscr{L}_{\bar g} = \ker\ (-\Delta_{\bar g} - n\lambda).
	\end{align*}
	This shows $\mathscr{L}_{\bar g}$ has non-trivial kernel if and only if $\bar g$ is spherical and $\ker \mathscr{L}_{\bar g}$ consisted of first eigenfunctions of $\Delta_{\bar g}$.
		
	For $\lambda = 0$, 
	\begin{align*}
	\mathscr{L}_{\bar g}u = \frac{1}{2} \Delta_{\bar g}^2 u.
	\end{align*}
	It is clear that $\mathscr{L}_{\bar g}$ is non-negative and $$\ker \mathscr{L}_{\bar g} = \ker  \Delta_{\bar g} = \mathbb{R}.$$
\end{proof}

\ \\

With these preparations, we summarize variational properties of $\mathcal{F}_{M,\bar g}$ at a strictly stable Einstein metric $\bar g$:
\begin{proposition}\label{prop:secvar_F}
Suppose $(M^n,\bar{g})$ is a strictly stable Einstein manifold with $$Ric_{\bar{g}}=(n-1)\lambda{\bar{g}},$$ where $\lambda \geq 0$ is a constant, then $\bar g$ is a critical point of $\mathcal{F}_{M,\bar g}$ and
$$ (D^2 \mathcal{F}_{M,\bar g}) \cdot (h, h) \leq 0$$
for any $h = \mathring h + \frac{1}{n}(tr_{\bar g} h) \bar{g}\in S_{2,\bar g}^{_{TT}}(M) \oplus (C^\infty (M) \cdot \bar g)$. Moreover, the equality holds if and only if 
\begin{itemize}
\item $h\in \mathbb{R}\bar{g}$, when $(M^n,\bar{g})$ is not isometric to the round sphere up to a rescaling of the metric. 
\item $h\in(\mathbb{R}\oplus E_{n\lambda})\bar{g}$, when $(M^n,\bar{g})$ is isometric to the round sphere $\mathbb{S}^{n}(r)$ with radius $r=\frac{1}{\sqrt{\lambda}}$, 
\end{itemize}
where $$ E_{n\lambda}:= \{u\in C^{\infty}(\mathbb{S}^{n}(r))\ |\ \Delta_{\mathbb{S}^{n}(r)}u + n\lambda u =0\}$$
is the space of first eigenfunctions for the spherical metric.
\end{proposition}

\begin{proof}
According to Proposition \ref{prop:critical_pt_J_Einstein}, we can conclude that $\bar g$ is a critical point of $ \mathcal{F}_{M,\bar g}$, since Einstein metrics are $J$-Einstein. Recall in Corollary \ref{cor:secvarEin}, we showed that
\begin{align*}
&(D^2 \mathcal{F}_{M,\bar g}) \cdot (h, h) \\
		\notag=&-2 \Vol_M(\bar g)^{\frac{4}{n}}  \left[ \int_M \langle \mathring h , (D \mathring{J}_{\bar g}) \cdot \mathring{h} \rangle_{\bar g} dv_{\bar g} + \frac{n+4}{4n^2} \int_M  \left[ ( tr_{\bar g} h - \overline{tr_{\bar{g}} h})   \mathscr{L}_{\bar{g}} (tr_{\bar g} h - \overline{tr_{\bar{g}} h}) \right] dv_{\bar g} \right]
	\end{align*}		
holds for any $h \in S_{2,\bar g}^{_{TT}}(M) \oplus (C^\infty (M) \cdot \bar g) $. It is obvious that $D^2\mathcal{F}_{M,\bar g}$ is non-positive definite according to Corollary \ref{cor:stable_Einstein_DJ_positive} and Proposition \ref{prop:positive_L}. Furthermore, $D^2 \mathcal{F}_{M,\bar g}$ vanishes if and only $\mathring h = 0$ and
\begin{align*}
	(tr_{\bar g} h - \overline{tr_{\bar{g}} h}) \in \ker \mathscr{L}_{\bar g}.
\end{align*} 
Now the conclusion follows from Proposition \ref{prop:positive_L}.
\end{proof}


\vskip .2in

Another fundamental result we need is the following version of \emph{Morse lemma} on Banach manifold for degenerate functions:
\begin{lemma}[Fisher-Marsden \cite{FM75}]\label{morselemma}
Let $\mathcal{P}$ be a Banach manifold and $f: \mathcal{P}\rightarrow \mathbb{R}$ a $C^2$ function. Suppose that $\mathcal{Q} \subset \mathcal{P}$ is a submanifold, $f=0$ and $df=0$ on $\mathcal{Q}$ and that there is a smooth normal bundle neighborhood of $\mathcal{Q}$ such that if $\mathcal{E}_x$ is the normal complement to $T_{x}\mathcal{Q}$ in $T_{x}\mathcal{P}$ then $d^2f(x)$ is weakly negative definite on $\mathcal{E}_x$ $($i.e. $d^2f(x)(v, v)\leq0$ with equality only if $v=0)$. Let $\langle \langle \ ,\ \rangle \rangle_{x}$ be a weak Riemannian structure with a smooth connection and assume that $f$ has a smooth $\langle \langle \ ,\ \rangle \rangle_{x}$-gradient, $Y(x)$.
Assume $DY(x)$ maps $\mathcal{E}_x$ to $\mathcal{E}_x$ and is an isomorphism for $x\in \mathcal{Q}$. Then there is a neighborhood $U$ of $\mathcal{Q}$ such that $y\in U$, $f(y)\geq0$ implies $y\in \mathcal{Q}$.\\
\end{lemma}

According to Theorem \ref{thm:slice_thm}, we can find a local slice $\mathcal{S}_{\bar g}$ through $\bar g$ and identify $\mathcal{Q}_{\bar{g}}$ to be the submanifold of $\mathcal{S}_{\bar g}$ consisted of homothetic metrics, that is,
$$\mathcal{Q}_{\bar{g}}:=\{c^2 \bar{g}\in\mathcal{S}_{\bar{g}} \ |\  c \neq 0\}.$$ Consider the restriction of $\mathcal{F}_{M,\bar g}$ on the local slice $\mathcal{S}_{\bar g}$, denoted by $\mathcal{F}_{M,\bar g} |_{\mathcal{S}_{\bar g}}$. Applying the previous Morse lemma, we obtain the following rigidity result:
\begin{proposition}\label{prop:rigidity_on_slice}
	Suppose $(M^n, \bar g)$ is a strictly stable Einstein manifold with Ricci curvature
	$$Ric_{\bar g} = (n-1)\lambda \bar g,$$
	where $\lambda \geq 0$ is a constant. There is a neighborhood of $\bar g$ in the local slice $\mathcal{S}_{\bar g}$, denoted by $U_{\bar g}$, such that any metric $g_s \in U_{\bar g}$ satisfying
	\begin{align*}
	\mathcal{F}_{M,\bar g}|_{\mathcal{S}_{\bar g}} (g_s) \geq \mathcal{F}_{M,\bar g}|_{\mathcal{S}_{\bar g}}(\bar g)
	\end{align*}
	implies that $g_s = c^2 \bar g$ for some constant $c > 0$.
\end{proposition}

\begin{proof}
From Proposition \ref{prop:secvar_F}, we conclude that $\bar g$ is a critical point of $\mathcal{F}_{M,\bar g} |_{\mathcal{S}_{\bar g}}$ and $D^2 \mathcal{F}_{M,\bar g} |_{\mathcal{S}_{\bar g}}$ is non-positive definite on $T_{\bar g} \mathcal{S}_{\bar g}$. Moreover, it is obvious that $D^2 \mathcal{F}_{M,\bar g} |_{\mathcal{S}_{\bar g}}$ is degenerate if and only if when restricted on 
\begin{align*}
T_{\bar{g}}\mathcal{Q}_{\bar g} = \mathbb{R} \bar{g}.
\end{align*}
Let $\mathcal{E}_{\bar g}$ be the $L^2$-orthogonal complement of $T_{\bar g}\mathcal{Q}_{\bar{g}}$ in $T_{\bar g} \mathcal{S}_{\bar g}$. By Theorem \ref{thm:slice_thm}, we can identify 
$$\mathcal{E}_{\bar g} = \left\{ h\in S_{2, \bar g}^{_{TT}}(M) \oplus (C^\infty(M)\cdot \bar g)\ \left| \ \int_M (tr_{\bar{g}}h) \ dv_{\bar{g}}= 0 \right.\right\},$$ if $\bar g$ is not spherical;\\
$$\mathcal{E}_{\bar{g}}=\left\{ h\in S_{2, \bar g}^{_{TT}}(M) \oplus (E_{n\lambda}^{\perp} \cdot \bar g)\ \left| \ \int_M (tr_{\bar{g}}h) \ dv_{\bar{g}}= 0 \right.\right\},$$ if $\bar g$ is spherical. Therefore, $D^2 \mathcal{F}_{M,\bar g} |_{\mathcal{S}_{\bar g}}$ is strictly negative definite on $\mathcal{E}_{\bar{g}}$.

We introduce a weak Riemannian structure 
$$ \langle\langle h, h \rangle\rangle_{g_{s}}:= \int_{M} [ \langle h,h\rangle_{g_{s}} + \langle\nabla_{g_s} h , \nabla_{g_s} h\rangle_{g_s} ]  dv_{g_s}=\int_{M} \langle(1-\Delta_{g_{s}})h , h\rangle_{g_s}  dv_{g_s}$$
on $\mathcal{S}_{\bar g}$. As in \cite{Yuan20}, it has a smooth connection and the $\langle \langle\  ,\  \rangle\rangle_{g_s}$-gradient of $\mathcal{F}_{M,\bar g}\big|_{\mathcal{S}_{\bar{g}}}$ is given by
$$Y(g_s)= P_{g_s}(1 - \Delta_{g_s})^{-1}\left[ \Vol_M(g_s)^\frac{4}{n}\left( \Gamma_{g_s}^{*}(\rho_{g_s}) + \frac{2}{n}g_{s} \Vol_M(g_s)^{-\frac{n+4}{n}}\mathcal{F}_{M,\bar g}(g_s)\right)\right],$$ 
where $P_{g_s}$ is the orthogonal projection to $T_{g_s}\mathcal{S}_{\bar g}$ and $\rho_{g_s} >0$ is a smooth function on $M$ satisfying $dv_{\bar{g}}= \rho_{g_s} dv_{g_s}$. Obviously, $Y(g_s)$ is a smooth vector field on $\mathcal{S}_{\bar{g}}$.  
Now we define an auxiliary vector field on $\mathcal{S}_{\bar g}$, $$Z(g_s):= \Vol_M(g_s)^\frac{4}{n}\left( \Gamma_{g_s}^{*}(\rho_{g_s}) + \frac{2}{n}g_{s} \Vol_M(g_s)^{-\frac{n+4}{n}}\mathcal{F}_{M,\bar g}(g_s)\right).$$
It is straightforward that $Z(\bar{g})=0$ due to the fact that $\bar{g}$ is Einstein and furthermore, 
$$(DZ_{\bar{g}}) \cdot h= (D^2\mathcal{F}_{M,\bar g}|_{\mathcal{S}_{\bar g}}) \cdot (h, \cdot)$$
for any $h\in \mathcal{E}_{\bar{g}}$. Thus we have
$$DY_{\bar{g}}= P_{\bar{g}}(1- \Delta_{\bar{g}})^{-1} (DZ_{\bar{g}}),$$ 
which implies $DY_{\bar{g}}$ is an isomorphism on $\mathcal{E}_{\bar g}$ due to the reason that $D^2\mathcal{F}_{M,\bar g}|_{\mathcal{S}_{\bar g}}$ is strictly negative definite on $\mathcal{E}_{\bar{g}}$ from previous discussions.

According to Lemma \ref{morselemma}, we can find a neighborhood $U_{\bar g} \subset \mathcal{S}_{\bar g}$ such that any metric $g_s \in U_{\bar g}$ satisfying $$\mathcal{F}_{M,\bar g}|_{\mathcal{S}_{\bar g}} (g_s) \geq \mathcal{F}_{M,\bar g}|_{\mathcal{S}_{\bar g}}(\bar g) = \mathcal{F}_{M,\bar g} (\bar g)$$ implies that $g_s \in \mathcal{Q}_{\bar g}$, which means we can find a constant $c > 0$ such that $g_s = c^2 \bar g$.
\end{proof}

\vskip .2in

Now we can prove our volume comparison theorem:
\begin{proof}[Proof of Theorem~\ref{thmvolQ}]
 Applying Theorem \ref{thm:slice_thm}, we can find a positive constant $\varepsilon_0 < \varepsilon_1$ such that for any metric $\hat g$ satisfies that
 \begin{align*}
 || \hat {g} - \bar{g}||_{C^4(M,\bar{g})} < \varepsilon_0,
 \end{align*}
 there exists a diffeomorphism $\varphi$ such that $\varphi^{*} \hat g\in U_{\bar g} \subseteq \mathcal{S}_{\bar{g}}$, where $U_{\bar g}$ is defined in Proposition \ref{prop:rigidity_on_slice}.

For $\lambda > 0$, suppose $g$ is a Riemannian metric on $M$ satisfying 
$$Q_{g}\geq Q_{\bar {g}}$$ 
and $$|| {g} - \bar{g}||_{C^4(M,\bar{g})} < \varepsilon_0,$$ but with \emph{reversed volume comparison}:
\begin{align}\label{ineq:reversed_vol_comp}
	 \Vol_M(g)\geq  \Vol_M(\bar{g}).
\end{align}
We are going to show that $g$ has to be isometric to $\bar g$ and hence the claimed volume comparison holds. 

According to the argument in the previous paragraph, there exists a diffeomorphism $\varphi$ such that $\varphi^{*}g\in U_{\bar{g}} \subseteq \mathcal{S}_{\bar g}$ and 
$$\mathcal{F}_{M, \bar g}|_{\mathcal{S}_{\bar g}}(\varphi^* g)=  \Vol_M(\varphi^*g)^{ \frac{4}{n}} \int_M (Q _g\circ\varphi)  dv_{\bar g}\geq  \Vol_M(\bar g)^{ \frac{4}{n}} \int_M Q (\bar g) dv_{\bar g}=\mathcal{F}_{M, \bar g}|_{\mathcal{S}_{\bar g}}(\bar g)$$
due to our assumptions and the fact that $Q_{\bar g}$ is a constant.
Thus, we conclude that $\varphi^{*}g=c^2\bar{g}$ for some positive $c \in \mathbb{R}$ by Proposition \ref{prop:rigidity_on_slice}. Now the reversed volume comparison \eqref{ineq:reversed_vol_comp} becomes
$$ \Vol_M(g)= \Vol_M(\varphi^*g)=c^n \Vol_M(\bar g) \geq  \Vol_M(\bar g),$$
which implies $c\geq 1$. However, the curvature comparison assumption implies $c \leq 1$, since
\begin{align}
	Q_{\bar g} = Q_{\bar g}\circ{\varphi} \leq Q_{g}\circ{\varphi}=Q_{\varphi^* g} =c^{-4}Q_{\bar g}.
\end{align}
Therefore, $\varphi^* g = \bar g$ and it concludes the theorem.
\end{proof}

\vskip .2in

With the same idea, we can prove the rigidity of strictly stable Ricci-flat manifolds:
\begin{proof}[Proof of Theorem \ref{thm:ricci_flat_rigidity}]
	Similar to the proof of Theorem \ref{thmvolQ}, we can find an $\varepsilon_0 > 0$ such that for any metric $\hat g$ satisfies
	\begin{align*}
	||\hat g - \bar g||_{C^4(M, \bar g)} < \varepsilon_0,
	\end{align*}
	there is a diffeomorphism $\varphi$ such that $\varphi^*\hat g \in U_{\bar g} \subset \mathcal{S}_{\bar g}$, where $U_{\bar g}$ is given by Proposition \ref{prop:rigidity_on_slice}.
	
	Suppose $g$ is a metric satisfying
	\begin{align*}
	Q_g \geq 0
	\end{align*}
	and
	\begin{align*}
	||g - \bar g||_{C^4(M, \bar g)} < \varepsilon_0,
	\end{align*}
	then we can find a diffeomorphism $\varphi$ such that $\varphi^* g \in U_{\bar g}$ and 
	\begin{align*}
		\mathcal{F}_{M, \bar g}|_{\mathcal{S}_{\bar g}}(\varphi^* g)=  \Vol_M(\varphi^*g)^{ \frac{4}{n}} \int_M (Q _g\circ\varphi)  dv_{\bar g}\geq 0 .
	\end{align*}
	However, the metric $\bar g$ is Ricci-flat and hence
	\begin{align*}
		\mathcal{F}_{M, \bar g}|_{\mathcal{S}_{\bar g}}(\bar g) = \Vol_M(\bar g)^{ \frac{4}{n}} \int_M Q (\bar g) dv_{\bar g} = 0.
	\end{align*}
	This shows $\varphi^* g = c^2 \bar g$ for some positive constant $c \in \mathbb{R}$ by Proposition \ref{prop:rigidity_on_slice}, which means $g$ is homothetic to $\bar g$ up to a diffeomorphism.
\end{proof}

\ \\

\section{Remarks and further discussions} \label{example}

In this section, we address some important observations and remarks regarding our main theorem. First, we make a comment on the stability assumption:
\begin{remark}\label{stab_ness}
The stability condition in Theorem \ref{thmvolQ} is necessary. This is the same phenomenon observed in \cite{Yuan20} for the volume comparison of scalar curvature. \\

Let $\bar g$ be the canonical product metric on $\mathbb{S}^2 \times \mathbb{S}^2 \times \mathbb{S}^2$. It is well-known that this manifold is unstable (c.f \cite{Kro14}). Consider the metric $$g_t = (1 + t^2)^{-1} g^1_{_{\mathbb{S}^2}} + (1-t)^{-1} g^2_{_{\mathbb{S}^2}} + (1+t)^{-1} g^3_{_{\mathbb{S}^2}}$$ with $t \in (0,1)$ sufficiently small.
		Then its $Q$-curvature is given by 
		 \begin{align*}
		 Q_{g_t} =& \frac{1}{50} \left( - 3 t^4  + 7 t^2  + 48 \right) > \frac{24}{25} = Q_{\bar g} 
		 \end{align*}
		 However, its volume satisfies 
		 $$ \Vol_M (g_t) = (1 - t^4)^{-1}  \Vol_M(\bar g)> \Vol_M(\bar g).$$
It shows that the volume comparison fails in this case. In fact, the volume comparison is not expected to hold for unstable Einstein manifolds due to Corollary \ref{cor:secvarEin}.
\end{remark}

\vskip.1in

Now we turn to the locality assumption. For general dimensions $n\geq 3$, we have proved a volume comparison result for metrics sufficiently closed to a strictly stable positive Einstein metric. It turns out that for dimension $n=4$, we do have a global volume comparison result as follows:
\begin{proposition}\label{prop:lcf_vol_comparison}
Let $(M^4, \bar{g})$ be a closed $4$-dimensional locally conformally flat Riemannian manifold with positive constant $Q$-curvature. Then for any metric $g$ on $M$ satisfies
$$Q_{g}\geq Q_{\bar{g}}$$
pointwisely on $M$, we have
$$ \Vol_M(g)\leq  \Vol_M(\bar{g})-\frac{1}{4Q_{\bar{g}}}||W_{g}||^2_{L^2(M,g)}\leq \Vol_M(\bar{g}),$$
where equality holds if and only if $g$ is also locally conformally flat.
\end{proposition}
\begin{proof}
	By Gauss-Bonnet-Chern formula,
	$$8\pi^2 \chi (M) = \int_M Q_g dv_g + \frac{1}{4}\int_M |W_g|^2 dv_g = \int_M Q_{\bar g} dv_{\bar g} = Q_{\bar g}  \Vol_M(\bar g) . $$
	Then
	\begin{align*}
		 \Vol_M(\bar g) =& Q_{\bar g}^{-1} \left(\int_M Q_g dv_g + \frac{1}{4}\int_M |W_g|^2 dv_g \right)\\
		\geq&  \Vol_M(g) + \frac{1}{4Q_{\bar g}} \int_M |W|_g^2 dv_g.
	\end{align*}
\end{proof}

\vskip .2in

As a straightforward application, we have
\begin{corollary}\label{cor:global_vol_comp_4-sphere_hyperbolic}
		Let $(M^4, \bar g)$ be either the standard round $4$-sphere or a closed hyperbolic $4$-manifold. Then for any metrics $g$ satisfies $$Q_g \geq Q_{\bar g},$$
		we have 
		$$ \Vol_M(g) \leq   \Vol_M(\bar g).$$
		Moreover, in case of the round $4$-sphere, the equality holds if and and only if $g$ is isometric to the spherical metric $\bar g$.
	\end{corollary}
	
	\begin{proof}
		We only need to prove the rigidity part of $4$-sphere. According to Proposition \ref{prop:lcf_vol_comparison}, the metric $g$ has to be locally conformally flat and hence $g \in [\bar g]$ by \emph{Kuiper's theorem}\cite{Kuiper49}, since $\mathbb{S}^4$ is simply connected. On the other hand, due to our assumption $Q_g \geq Q_{\bar g}$ and Gauss-Bonnet-Chern formula, we have $Q_g = Q_{\bar g}$. Now the conclusion follows from a uniqueness result in \cite{CY97, Lin98, Xu06}.
	\end{proof}

\vskip.1in

Based on this global volume comparison observed above for $4$-dimensional hyperbolic manifolds, we would like to propose the following conjecture:
\begin{conjecture}
For any $n\geq 3$, let $(M^n, \bar g)$ be a closed hyperbolic manifold. Suppose $g$ is a metric on $M$ with $$Q_g \geq Q_{\bar g},$$ then we have
		$$V_M(g) \leq V_M (\bar g).$$
\end{conjecture}

\begin{remark}
	This is a corresponding version of \emph{Schoen's conjecture} on scalar curvature (see \cite{Yuan20} for more details). As a first step, we would be interested in the question that whether this conjecture holds for metrics $C^4$-closed to the hyperbolic metric $\bar g$. In this case, it depends on a further research on the spectrum of the operator $\mathscr{L}_{\bar{g}}$. 
\end{remark}

\vskip .2in



\end{document}